\documentclass[preprint,12pt,3p,times]{elsarticle}
\usepackage{amssymb}
\usepackage{amssymb,amsmath,amsthm,amsfonts}
\usepackage[colorlinks,urlcolor=blue,citecolor=blue,linkcolor=blue]{hyperref}

\newtheorem{theorem}{Theorem}[section]

\newtheorem{lemma}[theorem]{Lemma}

\theoremstyle{definition}
\newtheorem{definition}[theorem]{Definition}
\theoremstyle{remark}

\numberwithin{equation}{section}
\allowdisplaybreaks

\journal{Elsevier}

\begin{document}

\begin{frontmatter}
\title{Existence of solutions for a class of Kirchhoff-type equations with indefinite potential}

\author{Linlian Xiao$^{1}$}

\author{Jiaqian Yuan$^{2}$}

\author{Jian Zhou\corref{cor1}
	$^{3}$}
\ead{zhoujiandepict@163.com}
\author{Yunshun Wu$^{4}$}

\affiliation[1]{organization={School of Mathematical Sciences, Guizhou Nromal University},
	city={Guiyang},
	postcode={550025}, 
	state={Guizhou},
	country={China}}

\cortext[cor1]{Corresponding author}

\begin{abstract}
In this paper, we consider the existence of solutions of the following Kirchhoff-type problem
\[
\left\{
\begin{array}
	[c]{ll}%
	-\left(a+b\int_{\mathbb{R}^3}|\nabla u|^2dx\right)\Delta u+ V(x)u=f(x,u),~{\rm{in}}~ \mathbb{R}^{3},\\
	u\in H^1(\mathbb{R}^3),%
\end{array} %
\right.
\]
where $a,b$ are postive constants, and the potential $V(x)$ is continuous and indefinite in sign. Under some suitable assumptions on $V(x)$ and $f$, we obtain the existence of solutions by the Symmetric Mountain Pass Theorem.

\end{abstract}


\end{frontmatter}
\section{Introduction and main result}
\label{1}
In this paper, we consider the existence of solutions of the following Kirchhoff-type problem	
\begin{equation}\label{e1-1}
	\left\{
	\begin{array}
		[c]{ll}%
		-\left(a+b\int_{\mathbb{R}^3}|\nabla u|^2dx\right)\Delta u+ V(x)u=f(x,u)\text{,}
		& \text{in }\mathbb{R}^{3}\text{,}\\
		u\in H^1(\mathbb{R}^3)\text{,}%
	\end{array} %
	\right.
\end{equation}
where $a,~b$ are postive constants, and the potential $V(x)$ is continuous and indefinite in sign. The nonlinear term $\int_{\mathbb{R}^3}|\nabla u|^2dx$ appears in (\ref{e1-1}), which means that (\ref{e1-1}) is not a pointwise identity. This leads to some mathematical difficulties that make the research particularly interesting. (\ref{e1-1}) has an interesting physics background. When $V(x)=0$, and a bounded domain $\Omega \subset\mathbb{R}^N$ is substituted $\mathbb{R}^3$, then we obtain the following nonlocal Kirchhoff-type problem
\begin{equation}\label{e1-1-1}
	\left\{
	\begin{array}
		[c]{ll}%
		-\left(a+b\int_{\Omega}|\nabla u|^2dx\right)\Delta u=f(x,u)\text{,}
		& \text{in }\Omega \text{,}\\
		u=0,  & \text{on }\Omega \text{.}%
	\end{array} %
	\right.
\end{equation}
The problem (\ref{e1-1-1}) is regard to the stationary analogue  of the equation
\begin{equation}\label{e1-1-2}
	u_{tt}-\left(a+b\int_{\Omega}|\nabla u|^2dx\right)\Delta u=f(x,u)\text{,}
\end{equation}
which was presented by Kirchhoff in \cite{G.Kirchhoff1883}, and (\ref{e1-1-2}) is a generalization of the classical D'Alembert's wave equation for free vibrations of elastic strings. Problem (\ref{e1-1-2}) has been increasingly more attention after Lions in \cite{Lions1978} introduced an abstract framework to the problem. We can refer to [\citealp{BernsteinS1940},\citealp{Arosio A1996},\citealp{Cavalcanti2001}] for the physical and mathematical background of this problem.

In recent years, Schr{\"o}dinger Kirchhoff equations have been extensively researched , there are massive works adopting various assumptions on $V(x)$ and $f$ see [\citealp {WuX2011}-\citealp{Jiang2023}]. The potential $V(x)$ is assumed to be positive definite has been considered in [\citealp{WuX2011}-\citealp{Zhang2020}]. In \cite{WuX2011}, Wu used a Symmetric Montain Theorem obtained nontrivial solutions and high energy solutions for equations similar to (\ref{e1-1}) in $\mathbb{R}^{N}$. In \cite{Chengb2016}, by Ekeland's variational principle and the Montain Pass Theorem, Cheng obtained multiplicity of nontrivial solutions for the nonhomogeneous Schr{\"o}dinger Kirchhoff type problem in $\mathbb{R}^{N}$. The potential $V(x)$ is indefinite has been considered in [\citealp{ShuaiW2015}-\citealp{Jiang2023}]. In \cite{Wu2015}, Chen and Wu got a nontrivial solution and an unbounded sequence of solutions for the problem (\ref{e1-1}) in $\mathbb{R}^{N}$ via the Morse Theory and the Fountain Theorem. In \cite{Jiangs2022}, using the Local Linking Theorem and Clark's Theorem, Jiang and Liu obtained the  existence of multiple solutions for problem (\ref{e1-1}). 

In this paper, we will consider $V(x)$ is indefinite in sign and do not assume any compactness condition on $V(x)$ which is different from most of the articles mentioned above. Motived by Chen \cite{Chen2021} and Sun  \cite{SunJ2015}, we overcome two difficulties, namely, verifying the link geometry and the boundedness of Cerami sequence for the corresponding functional of (\ref{e1-1}). We obtain the existence of solutions for (\ref{e1-1}) by the Symmetric Montain Pass Theorem.

Set $F(x,u)=\int_0^uf(x,s)ds$. $V^+(x)={\rm{max}}\left\lbrace V(x),0\right\rbrace $, $V^-(x)={\rm{max}}\left\lbrace -V(x),0\right\rbrace $.
Before stating our main result, we make the following assumptions:
\begin{enumerate}	
	\item[$\left(  V1\right)  $] $V(x)\in C(\mathbb{R}^{3},\mathbb{R})$ with $V(x)=V^+(x)-V^-(x)$ and $V(x)$ is bounded from below, and there is $M>0$ such that the set $\left\lbrace x\in\mathbb{R}^{3}|V^+(x)<M \right\rbrace$ is nonempty and has finite measure.	
	\item[$\left(  V_2\right)  $] There exists a constant $\eta_0>1$ such that 
	\[
	\eta_1:=\inf\limits_{u\in H^1(\mathbb{R}^{3})\setminus \left\lbrace 0\right\rbrace } \frac{\int_{\mathbb{R}^{3}} \left(  a\left | \nabla u \right|^2+V^+u^2 \right) dx }{\int_{\mathbb{R}^{3}}V^-u^2dx}\geq \eta_0.
	\]
	\item[$\left(  f_{1}\right)  $] $f\in C^1(\mathbb{R}^{3},\mathbb{R})$, and there exist constants $p\in(2,6)$ and $c>0$ such that
	\[
	\left| f(x,u)\right| \leq c(1+\left| u\right|^{p-1} ),~~~~\forall (x,u)\in \mathbb{R}^{3} \times\mathbb{R}.
	\]
	\item[$\left(  f_{2}\right)  $]  $f(x,u)=o(u)$ as $u\rightarrow0$ uniformly in $x\in\mathbb{R}^3$, and is 4-superlinear at infinity,
	\[
	\lim_{|u|\rightarrow\infty}\frac{F(x,u)}{u^4}=+\infty.
	\]	
	\item[$\left(  f_{3}\right)  $] There exist $a_0,b_0>0$ and $\alpha\in(0,\alpha_*)$ such that 
	\[
	0<(4+\frac{1}{a_0\left| u\right|^\alpha +b_0})F(x,u)\leq uf(x,u)
	, ~~~{\rm{for}} ~x\in\mathbb{R}^{3}~{\rm{and}} ~u\neq0,\]
	where $\alpha_* :={\rm{min}}\left\lbrace p',5p'-6\right\rbrace $, $\frac{1}{p}+\frac{1}{p'}=1$.	
	\item[$\left(  f_{4}\right)  $] $\lim\limits_{|x|\rightarrow\infty}\sup\limits_{|u|\leq l}\frac{\left| f(x,u)\right| }{\left| u\right| }=0$ for every $l>0$.		
\end{enumerate}	

Now, we are ready to state the main result of this paper:	
\begin{theorem}\label{th1}
	Under assuptions $(V_1),~(V_2)$ and $(f_1)-(f_4)$, if $f(x,u)$ is odd in $u$, then problem (\ref{e1-1}) possesses infinitely many solutions.	
\end{theorem}
\section{Preliminaries}	
\label{2}
We work in the Hilbert space	
\[
E:=\left\lbrace u\in H^1(\mathbb{R}^{3}):	\int_{\mathbb{R}^{3}} \left( a \left| \nabla u \right|^2+V^+(x)\left| u\right| ^2\right) dx<+ \infty\right\rbrace ,
\]	
with the inner product	
\[
\left\langle u,v\right\rangle =\int_{\mathbb{R}^{3}}  \left( a \nabla u \nabla v +V^+(x) uv \right) dx,~~\forall u,v\in E,
\]	
and the norm	
\[
\left\| u\right\| =\left( \int_{\mathbb{R}^{3}} \left( a \left| \nabla u \right|^2+V^+(x)\left| u\right| ^2\right) dx\right) ^{1/2},~~\forall u\in E.
\]

The problem (\ref{e1-1}) has a variational structure, then a weak solution of problem (\ref{e1-1}) is a critical point of the following functional $\Phi:E\rightarrow\mathbb{R}$
\begin{equation}\label{e2-1}
	\Phi(u)=\frac{1}{2}\int_{\mathbb{R}^3} \left( a |  \nabla u|^2+ V(x)u^2\right) dx+\frac{b}{4}\left(
	\int_{\mathbb{R}^3}|\nabla u|^2dx\right)^2-\int_{\mathbb{R}^3}F(x,u)dx.
\end{equation}
Then under the assumptions $(V_1)$, $(f_1)$ and $(f_2)$, the functional $\Phi\in C^1(E,\mathbb{R})$ and for all $u,v\in E$,
\begin{equation}\label{e2-2}
	\left\langle \Phi'(u),v\right\rangle =\int_{\mathbb{R}^3} \left( a \nabla u \nabla v+ V(x)uv\right) dx+b
	\int_{\mathbb{R}^3}|\nabla u|^2dx\int_{\mathbb{R}^3}\nabla u\nabla vdx-\int_{\mathbb{R}^3}f(x,u)vdx.	
\end{equation}

For any $s\in [2,6]$, since the embedding $E\hookrightarrow L^s(\mathbb{R}^3)$ is continuous, there exists a constant $d_s>0$ such that
\begin{equation}\label{e2-3}    
	|u|_s\leq d_s \|u\|,~~\forall u\in E.
\end{equation}

Forthermore, it follows from $(V_2)$ that
\begin{equation}\label{e2-4} 
	\begin{split}	
		\int_{\mathbb{R}^{3}} \left( a \left| \nabla u \right|^2+V^+\left| u\right| ^2 \right) dx
		\geq& \int_{\mathbb{R}^{3}} \left( a \left| \nabla u \right|^2+V\left| u\right| ^2\right) dx \\
		\geq& \frac{\eta_0-1}{\eta_0} \int_{\mathbb{R}^{3}} \left( a \left| \nabla u \right|^2+V^+\left| u\right| ^2\right) dx.
		\end{split}
\end{equation}	

To complete the proof of theorem \ref{th1}, we need the following Symmetric Mountain Pass Theorem: 	
\begin{theorem}\label{th2} (\cite{Rabinowitz1986})  
	Let X be an infinite demensional Banach space, $X=Y \oplus Z$, where Y is finite dimensional. If $I \in	C^1(X,\mathbb{R})$ satisfies $(C)_{c}$-condition for all $c>0$, and
	\begin{enumerate}
		\item[$\left(  I_1\right)  $] $I(0)=0,I(-u)=I(u)$, $\forall u \in X$;
		\item[$\left(  I_2\right)  $] there exist constants $\alpha$, $\rho$ $> 0$, such that $I|_{\partial B_{\rho}\cap Z}\ge \alpha$;	
		\item[$\left(  I_3\right)  $] for any finite dimensional subspace $\tilde{X}\subset X$, there is $R=R(\tilde{X})>0$, such that $I(u)\le 0$ on $\tilde{X} \setminus B_{R}$;\\	
		then $I$ possesses an unbounded sequence of critical values.
	\end{enumerate}	
\end{theorem}
\begin{definition}\label{3-1}	Assume $E$ be a Banach space, and $\Phi \in C^1(E,\mathbb{R}^{3})$. For given $c \in \mathbb{R}$, a sequence $\left\lbrace u_n \right\rbrace \subset E $ is called a Cerami sequence of $\Phi$ at a level $c$ (shortly, $(C)_{c}$ sequence) if 
	\begin{equation}\label{e3-1}
		\Phi(u_n)\rightarrow c,~~~~(1+\left\| u_n\right\| )\left\|\Phi'(u_n) \right\| \rightarrow 0.
	\end{equation}	
	We say that $\Phi$ satisfies the Cerami condition at level $c$ (shortly, $(C)_{c}$-condition) if every $(C)_{c}$ sequence of $\Phi$ contains a convergent subsequence. If $\Phi$ satisfies  $(C)_{c}$-condition for every $c \in \mathbb{R}$, then we say that $\Phi$ satisfies the Cerami condition (shortly, $(C)$-condition ). 
\end{definition}
\section{Proof of main results}
\label{3}		
\begin{lemma}\label{3-2}
	Suppose that $(V_1)$, $(f_1)$, $(f_2)$ and $(f_3)$ are satisfied and $c\in\mathbb{R}$. Then any $(C)_c$ sequence of $\Phi$ is bounded.
\end{lemma}	
\begin{proof}
	It is follows from $(f_3)$ that, for all $u\neq0$ and $x\in \mathbb{R}^{3}$,
	\[
	uf(x,u)-4F(x,u)\geq \frac{1}{4a_0\left| u\right|^{\alpha} +4b_0+1}uf(x,u) >0.
	\]
	Let $\left\lbrace u_n\right\rbrace$ be a $(C)_c$ sequence of $\Phi$, that is, a sequence satisfying (\ref{e3-1}). Set $\Omega_n:=\left\lbrace  x\in\mathbb{R}^{3} : \left| u_n(x) \right| <1 \right\rbrace $ and $\Omega_n^c:=\mathbb{R}^{3}\setminus \Omega_n$. Then there are constants $c_1,c_2>0$ such that
	\[
	4a_0\left| u_n\right|^{\alpha} +4b_0+1 \leq 1/c_1, ~~\forall x\in \Omega_n,
	\]
	and
	\[
	4a_0\left| u_n\right|^{\alpha} +4b_0+1 \leq \left| u_n\right|^{\alpha}/c_2, ~~\forall x\in \Omega_n^c.
	\]
	For $n$ sufficient large, there exists $M_1>0$, such that 
	\begin{equation}\label{e3--2}
		\begin{split}	 
			M_1&\geq 4\Phi(u_n)-\left\langle \Phi'(u_n), u_n \right\rangle \\
			&=\int_{\mathbb{R}^3} \left(a |\nabla u_n|^2+ V(x)u_n^2 \right) dx+\int_{\mathbb{R}^3}\left( u_nf\left( x,u_n\right) -4F\left( x,u_n\right) \right) dx\\
			&\geq \frac{\eta_0-1}{\eta_0}\left\| u_n\right\|^2  +\int_{\mathbb{R}^3}\left( u_nf\left( x,u_n\right) -4F\left( x,u_n\right) \right) dx\\ 
			&\geq  \int_{\mathbb{R}^3}\left( u_nf(x,u_n)-4F(x,u_n)\right) dx\\ 
			&\geq \int_{\mathbb{R}^3}\frac{u_nf(x,u_n)}{4a_0\left| u_n\right|^{\alpha} +4b_0+1}dx\\
			&\geq c_1\int_{\Omega_n}u_nf(x,u_n)dx+c_2\int_{\Omega_n^c}\left| u_n\right|^{-\alpha}u_nf(x,u_n)dx.
		\end{split}
	\end{equation}
	Note that $\alpha<5p'-6$ by $(f_3)$. We have
	\[
	\frac{1}{p'}<\frac{6}{5p'}<\frac{6}{6+\alpha} ~~{\rm{and}}~~\frac{2}{2+\alpha}<\frac{6}{6+\alpha}.
	\]
	Then we can chose a constant $r\in (0,1)$ such that
	\begin{equation}\label{e3-3}
		{\rm{max}}\left\lbrace \frac{6}{5p'},   \frac{2}{2+\alpha}  \right\rbrace \leq r \leq \frac{6}{6+\alpha}.
	\end{equation}
	Let $s:=r/(1-r)>0$. Then $\frac{1}{r}+\frac{1}{-s}=1$. By (\ref{e3--2}) and the inverse H{\"o}lder inequality we have 	
	\begin{equation}\label{e3-4}
		\begin{split}	 
			M_1&\geq c_1\int_{\Omega_n}u_nf(x,u_n)dx+c_2\int_{\Omega_n^c}\left| u_n\right|^{-\alpha}u_nf(x,u_n)dx\\
			&\geq c_1\int_{\Omega_n}u_nf(x,u_n)dx+c_2\left( \int_{\Omega_n^c}\left(u_nf(x,u_n) \right)^rdx \right) ^{1/r}\left( \int_{\Omega_n^c} \left| u_n\right|^{\alpha s}dx \right) ^{1/(-s)}\\		
			&\geq c_1\int_{\Omega_n}u_nf(x,u_n)dx+c_2\frac{\left( \int_{\Omega_n^c}\left(u_nf(x,u_n) \right)^r dx\right) ^{1/r}}{\left|  u_n\right| ^{\alpha }_{\alpha s}}.	
		\end{split}
	\end{equation}
	By $(f_1)$ and $(f_2)$ we have 
	\begin{align*}	
		&\left| f(x,u) \right| ^{p'r}\leq \left( c_3 \left| u\right|^ {(p-1)(p'-1)} \left| f(x,u) \right| \right) ^r=c_4\left( uf(x,u) \right) ^r,~~\forall~\left| u\right| \geq 1,\\
		&\left| f(x,u) \right| ^{2}\leq	c_5\left| u\right| \left| f(x,u)\right| =c_5uf(x,u),~~\forall~\left| u\right| <1.
	\end{align*}
	Therefore by (\ref{e3-4}) we have 
	\begin{equation}\label{e3-5}  
		\left( \int_{\Omega_n^c} \left| f(x,u_n) \right| ^{p'r} dx\right) ^{1/{p'r}}\leq  c_6 \left|  u_n\right| ^{\alpha/{p'} }_{\alpha s},	
	\end{equation}		
	\begin{equation}\label{e3-6} 
		\left( \int_{\Omega_n} \left| f(x,u_n) \right| ^{2} dx\right) ^{1/{2}}\leq  c_7.	
	\end{equation}
	
	In view of (\ref{e3-3}), we easily check that $p'r>1$ , $\alpha s \in[2,6]$ and $(p'r)'\in\left( 2,6\right] $,  where $(p'r)'={p'r}/\left( {p'r-1}\right) $. Consequently, by (\ref{e3-5}), (\ref{e3-6}) and the H{\"o}lder inequality, the Sobolev inequality, for $n$ large enough,
	\begin{align*}	
		\int_{\mathbb{R}^3}\left(a |\nabla u_n|^2+ V(x)u_n^2\right) dx=&\left\langle \Phi'(u_n),u_n \right\rangle -b\left(
		\int_{\mathbb{R}^3}|\nabla u_n|^2dx\right)^2+\int_{\mathbb{R}^3}f(x,u_n)u_ndx\\
		\leq& \left\| u_n \right\| +\int_{\mathbb{R}^3}f(x,u_n)u_ndx\\
		\leq& \left\| u_n \right\|+\left( \int_{\Pi_n}\left| f(x,u_n)\right|^2dx\right)^{1/2}  \left| u_n\right| _2\\
		&+\left( \int_{\Pi_n^c}\left| f(x,u_n)\right|^{p'r}dx\right)^{1/{p'r}}  \left| u_n\right| _{(p'r)'}\\
		\leq& \left\| u_n \right\|+c_7\left| u_n\right| _2+c_6\left| u_n \right|^{\alpha/{p'}}_{\alpha s} \left| u_n \right|_{(p'r)'}\\
		\leq& c_8\left\| u_n \right\|+c_9\left\| u_n \right\|\left\| u_n \right\|^{\alpha/{p'}}.
	\end{align*}	
	where $c_8$, $c_9$,$c_{10}$, $c_{11}>0$ are some constants.\\
	Therefore by (\ref{e2-4}) we have 
	\[
	\left\| u_n \right\| \leq c_{10}+c_{11} \left\| u_n \right\|^{\alpha/{p'}}.
	\]
	Note that $\alpha<p'$. Then we easily verify that $\left\lbrace u_n\right\rbrace$ is bounded.
\end{proof}
\begin{lemma}\label{3--3}  
	Suppose that $(V_1)$ and $(f_1)-(f_4)$ are satisfied. Then $\Phi$ satisfies $(C)_{c}$-condition.
\end{lemma}
\begin{proof}
	From Lemma \ref{3-2} we know that any $(C)_c$ sequence $\left\lbrace  u_n\right\rbrace $ is bounded in $E$. Then, passing to a subsequence, we may assume that $u_n\rightharpoonup u$ in $E$ and $u_n \rightarrow u$ in $L^s_{loc}(\mathbb{R}^{3})$, $s\in \left[ 2,6 \right)  $.\\
	Note that, by (\ref{e2-2})
	\begin{align*}	
		\left\langle  \Phi'(u_n),u_n-u\right\rangle =&\int_{\mathbb{R}^3} \left( a \nabla u_n \nabla (u_n-u)+ V(x)u_n (u_n-u)\right)  dx\\
		&+b
		\int_{\mathbb{R}^3}|\nabla u_n|^2dx\int_{\mathbb{R}^3}\nabla u_n\nabla (u_n-u)dx-\int_{\mathbb{R}^3}f(x,u_n)(u_n-u)dx.\\
		=&\int_{\mathbb{R}^3}\left(a \left| \nabla u_n \right| ^2+ V(x)u_n^2\right) dx -\int_{\mathbb{R}^3}\left(a \nabla u_n \nabla u+ V(x)u_n u \right) dx \\
		&+b
		\int_{\mathbb{R}^3}|\nabla u_n|^2dx\int_{\mathbb{R}^3}\nabla u_n\nabla (u_n-u)dx-\int_{\mathbb{R}^3}f(x,u_n)(u_n-u)dx.\\
		=&\int_{\mathbb{R}^3} \left(a \left| \nabla u_n \right| ^2+ V^+u_n^2\right) dx-\int_{\mathbb{R}^3} \left(a \nabla u_n \nabla u+ V^+u_n u \right) dx\\
		& -\int_{\mathbb{R}^3} V^-u_n^2 dx+\int_{\mathbb{R}^3} V^-u_nu dx-\int_{\mathbb{R}^3}f(x,u_n)(u_n-u)dx.\\
		&+b\int_{\mathbb{R}^3}|\nabla u_n|^2dx\int_{\mathbb{R}^3}\nabla u_n\nabla (u_n-u)dx\\
		=&\left\langle u_n,u_n-u\right\rangle -b\int_{\mathbb{R}^3}|\nabla u_n|^2dx\int_{\mathbb{R}^3}\nabla u_n\nabla (u-u_n)dx\\
		&-\int_{\mathbb{R}^3} V^-\left( u_n^2-u_n u \right)dx-\int_{\mathbb{R}^3}f(x,u_n)(u_n-u)dx ,
	\end{align*}	
	we have 
	\begin{equation}\label{e3-7}
		\begin{split}
			0\leq &\lim\sup_{n\rightarrow \infty}\left( \left\| u_n\right\|^2 - \left\| u\right\|^2 \right)=\lim\sup_{n\rightarrow \infty}\left\langle u_n,u_n-u\right\rangle   \\
			= &\lim\sup_{n\rightarrow \infty} [\left\langle \Phi'(u_n),u_n-u\right\rangle+b\int_{\mathbb{R}^3}|\nabla u_n|^2dx\int_{\mathbb{R}^3}\nabla u_n\nabla (u-u_n)dx\\
			&+\int_{\mathbb{R}^3} V^-u_n \left( u_n- u \right)dx+\int_{\mathbb{R}^3}f(x,u_n)(u_n-u)dx].
		\end{split}
	\end{equation}
	From (\ref{e3-1})
	\begin{equation}\label{e3-8}
		\left\langle \Phi'(u_n),u_n-u\right\rangle \rightarrow 0,~~{\rm{as}}~~n\rightarrow \infty.
	\end{equation}
	Since $u_n \rightharpoonup u $ in $E$, we know that $\int_{\mathbb{R}^3}\nabla u\nabla (u_n-u)dx\rightarrow0$ as $n\rightarrow\infty$. Consequently, by the boundedness of $\left\lbrace u_n\right\rbrace $ in $E$, we have 
	\begin{equation}\label{e3-9}
		b \int_{\mathbb{R}^3}|\nabla u_n|^2dx \int_{\mathbb{R}^3}\nabla u_n\nabla (u-u_n)dx\rightarrow0	,~~{\rm{as}}~~n\rightarrow\infty.
	\end{equation}
	Noting that $V^-(x)\geq 0$ for all $x\in \mathbb{R}^{3}$ and $(V_1)$ implies that $ V^-\in L^\infty (\mathbb{R}^{3})$. Moreover, it follows from ($V_1$) that $\left\lbrace V^+=0 \right\rbrace $ has finite measure, which implies that $\left\lbrace V^-(x)>0 \right\rbrace $ has finite measure. Since $u_n\rightharpoonup u$ in $E$ and $u_n \rightarrow u$ in $L^s_{loc}(\mathbb{R}^{3})$, $s\in \left[ 2,6 \right)  $, we have
	\begin{equation}\label{e3-10}
		\begin{split}
			\int_{\mathbb{R}^3} V^-u_n(u_n-u)dx&= \left| \int_{suppV^-} V^-u_n(u_n-u)dx\right| \\
			&\leq \left\| V^-\right\|_{\infty}\int_{suppV^-} \left| u_n\right|\left| u_n-u\right| dx \\
			&\leq \left\| V^-\right\|_{\infty} \left( \int_{suppV^-}\left| u_n\right| ^2dx\right) ^{1/2}\left( \int_{suppV^-}\left| u_n-u\right| ^2dx\right) ^{1/2}\\
			&\rightarrow 0,~~{\rm{as}}~~n\rightarrow \infty.
		\end{split}
	\end{equation}	
	Next, let $\varepsilon>0$, for $l\geq 1$, it follows from ($f_1$) and H{\"o}lder inequality that
	\begin{align*}
		\int_{\left| u_n\right| \geq l}f(x,u_n)(u_n-u)dx&\leq 2c\int_{\left| u_n\right| \geq l}\left| u_n \right|^{p-1}\left| u_n-u\right| dx\\
		&\leq 2cl^{p-6}\int_{\left| u_n\right| \geq l}\left| u_n \right|^{5}\left| u_n-u\right| dx\\
		&\leq 2cl^{p-6}\left| u_n \right|^{5}_{6}\left| u_n-u\right|_{6},	
	\end{align*}	
	since $p<6$, we may fix $l$ large enough such that
	\begin{equation}\label{e3-11}	
		\int_{\left| u_n\right| \geq l}f(x,u_n)(u_n-u)dx \leq \frac{\varepsilon}{3},	
	\end{equation}	
	for all $n$. Moreover, by ($f_4$) there exists $L>0$ such that
	\begin{equation}\label{e3-12}	
		\int_{\left| u_n\right| \leq l , \left| x\right| \geq L}f(x,u_n)(u_n-u)dx \leq \left| u_n\right| _2\left|u_n-u \right|_2 \sup_{|u_n|\leq l, \left| x\right| \geq L}\frac{\left| f(x,u_n)\right| }{\left| u_n\right| }    \leq \frac{\varepsilon}{3},
	\end{equation}
	for all $n$.
	For any $\varepsilon>0$, by $(f_1)$ and $(f_2)$, there exists $C_\varepsilon>0$ such that
	\begin{equation}\label{e3-13}
		\left| f \left( x,u\right) \right| \leq \varepsilon\left| u\right| +C_\varepsilon\left| u\right|^{p-1},~~ \forall(x,u)\in \mathbb{R}^{3} \times \mathbb{R} ,\\	
	\end{equation}	
	and
	\begin{equation}\label{e3-14}
		\left| F \left( x,u\right) \right| \leq \frac{\varepsilon}{2}   \left| u\right|^2 + \frac{C_\varepsilon}{p} \left| u\right|^{p},~~ \forall(x,u)\in \mathbb{R}^{3} \times \mathbb{R} ,
	\end{equation}	
	where $2< p <6$.
	Since $u_n \rightarrow u$ in $L^s(B_L(0))$ for $s\in \left[ 2,6 \right)$, from  (\ref{e3-13}) we have
	\begin{equation}\label{e3-15}
		\begin{split}	
			\int_{\left| u_n\right| \leq l , \left| x\right| \leq L}f(x,u_n)(u_n-u)dx\leq & \left(  \varepsilon+C_{\varepsilon}\right)   \int_{\left| u_n\right| \leq l,\left| x\right| \leq L} \left( \left| u_n\right|+\left| u_n\right|^{p-1} \right)  \left|u_n-u \right|dx\\
			\leq& \left( \varepsilon+C_{\varepsilon}\right)  \left| u_n\right|_2 \left|u_n-u \right|_{L^2(B_L(0))}\\
			&+\left( \varepsilon+C_{\varepsilon}\right)\left| u_n\right|_p^{p-1}\left|u_n-u \right|_{L^p(B_L(0))} \\
			\leq& \frac{\varepsilon}{3}	,
		\end{split}	
	\end{equation}	
	for $n$ large enough. Combining	(\ref{e3-11}), (\ref{e3-12}) (\ref{e3-15}), we conclude that
	\begin{equation}\label{e3-16}
		\int_{\mathbb{R}^3}f(x,u_n)(u_n-u)dx\leq \varepsilon, 
	\end{equation}	
	for $n$ large enough. Since $\varepsilon $ is arbitrary, (\ref{e3-16}), together with (\ref{e3-7})-(\ref{e3-10}), we get $\left\| u_n\right\| \rightarrow  \left\| u\right\| $. Thus, $u_n\rightarrow u$ in $E$.	
\end{proof}
{\bf Proof of Theorem \ref{th1}} 

Let $\left\lbrace e_j\right\rbrace $ is a total orthonormal basis of $E$ and define $X_j=\mathbb{R}e_j,$
\[  
Y_k={\bigoplus}_{j=1}^kX_j,~~~~~Z_k={\bigoplus}_{j=k+1}^\infty X_j,~~k\in \mathbb{Z}.
\]
\begin{proof}
	Obviously, $\Phi(0)=0$ and  $\Phi$ is even due to $f$ is odd, we will verify that $\Phi$ satisfies the remain conditions of Theorem \ref{th2}.
	
	Firstly, we can verify that $\Phi$ satisfies ($I_2$). By (\ref{e2-4}) and (\ref{e3-14}) with $0<\varepsilon<
	\frac{\eta_0-1}{2\eta_0d_2^2}$, we have	
	\begin{align*}
		\Phi(u)&=\frac{1}{2}\int_{\mathbb{R}^3} \left(a |\nabla u|^2+ V(x)u^2\right) dx+\frac{b}{4}\left(
		\int_{\mathbb{R}^3}|\nabla u|^2dx\right)^2-\int_{\mathbb{R}^3}F(x,u)dx \\
		&\ge \frac{\eta_0-1}{2\eta_0}\left\|u \right\|^2 -\int_{\mathbb{R}^{3}}F(x,u)dx \\
		&\ge \frac{\eta_0-1}{2\eta_0}\left\|u \right\|^2- \frac{\varepsilon}{2}  \left| u\right|^2 _2-\frac{C_\varepsilon}{p} \left| u\right|^{p}_p \\
		&\ge \frac{1}{2} \left( \frac{\eta_0-1}{\eta_0}-\varepsilon d_2^2 \right) \left\|u \right\|^2 -\frac{C_\varepsilon}{p} d^p_p  \left\| u\right\|^{p}\\
		&\ge  \frac{1}{4} \frac{\eta_0-1}{\eta_0} \left\|u \right\|^2 -\frac{C_\varepsilon}{p} d^p_p  \left\| u\right\|^{p},
	\end{align*}
	for all $u\in \partial{B_\rho}$, where $B_\rho=\left\lbrace u\in E:\left\| u\right\| < \rho   \right\rbrace $. Therefore,
	\[
	\Phi|_{\partial B_{\rho}\cap Z_k}\geq \frac{1}{4} \frac{\eta_0-1}{\eta_0} \rho^2 -\frac{C_\varepsilon}{p} d^p_p  \rho^{p}:= \alpha>0,
	\]
	for $\rho$ small enough.	
	
	Secondly, we verify that $\Phi$ satisfies ($I_3$), for any finite dimensional subspace $\tilde{E}\subset E$, there exists a positive intergral number $m$ such that $\tilde{E}\subset E_m$. Since all norms are equivalent in a finite dimensional space, there is a constant $b_1>0$ such that
	\[
	\left| u \right|_4\geq b_1\left\| u\right\| ,~~\forall u\in E_m.
	\]
	By $(f_1)$ and $(f_2)$ we know that for any $M_2>\frac{b}{4b_1^4}$, there is a constant $C(M_2)>0	$ such that 
	\[
	F(x,u)\geq M_2\left| u \right| ^4 -C(M_2)\left| u \right|^2,~~\forall(x,u)\in \mathbb{R}^{3}\times \mathbb{R}.
	\]
	Hence 
	\begin{align*}	
		\Phi(u)&\leq \frac{1}{2}\left\|u \right\|^2+ \frac{b}{4}\left\|u \right\|^4-M_2\left| u \right| ^4_4 +C(M_2)\left| u \right|^2_2\\
		&\leq \frac{1}{2}\left\|u \right\|^2+ \frac{b}{4}\left\|u \right\|^4-M_2b_1^4\left\| u \right\| ^4 +C(M_2)d^2_2\left\| u \right\|^2\\
		&=(\frac{1}{2}+C(M_2)d^2_2)\left\| u \right\|^2-(M_2b_1^4-\frac{b}{4})\left\|u \right\|^4,~~\forall u\in E_m.
	\end{align*}	
	Consequently, there is a large 	$R=R(\tilde{E})>0$ such that $	\Phi(u)\leq 0$ on $\tilde{E}\setminus B_{R}$.
		
	From Lemmas \ref{3-2} and \ref{3--3}, $\Phi$ satisfies $(C)_{c}$-condition, by Theorem \ref{th2} problem (\ref{e1-1}) possesses infinitely many solutions.				
\end{proof}	

\section*{Acknowledgements}
The authors would like to thank the unknown referee for his/her valuable comments and suggestions.

\section*{Funding}
This work was supported by National Natural Science Foundation of China (No. 12161019), Guizhou Provincial Basic Research Program (Natural Science) (No. QKHJC-ZK [2022] YB 318), Natural Science Research Project of Guizhou Provincial Department of Education (No. QJJ [2023] 011), Academic Young Talent Fund of Guizhou Normal University (No. QSXM [2022] 03).

\section*{Abbreviations}
Not applicable.

\section*{Availability of data and materials}
Not applicable.

\section*{Competing interests}
The authors declare that they have no competing interests.

\section*{Authors’ contributions}
The first author writes and revises this paper, the second author checks and proofreads this paper, and the third and fourth authors suggest changes to this paper. All authors read and approved the final manuscript.


\end{document}